\documentclass{amsart}
\usepackage{graphicx}
\usepackage{amssymb}
\usepackage{amscd}
\usepackage{cite}
\usepackage[all]{xy}
\usepackage{url}
\newtheorem{tw}{Theorem}[section]
\newtheorem{prop}[tw]{Proposition}

\newtheorem{wn}[tw]{Corollary}

\theoremstyle{remark}
\newtheorem{uw}[tw]{Remark}

\theoremstyle{definition}

\newcommand{\cal}[1]{\mathcal{#1}}

\newcommand{\ro}{\varrho}

\newcommand{\ukre}[1]{\underline{#1}}

\newcommand{\map}[3]{#1\colon #2\to #3}

\newcommand{\cM}{{\cal M}}

\newcommand{\SR}[1]{\underset{\rightarrow}{[#1]}}
\newcommand{\SL}[1]{\underset{\leftarrow}{[#1]}}

\begin{document}

\numberwithin{equation}{section}

\title[
A finite presentation for the twist subgroup\ldots]
{A finite presentation for the mapping class group of a nonorientable surface with Dehn twists and 
one crosscap slide as generators}

\author{Micha\l\ Stukow}

\thanks{Suported by NCS grant 2012/05/B/ST1/02171.}

\address[]{
Institute of Mathematics, University of Gda\'nsk, Wita Stwosza 57, 80-952 Gda\'nsk, Poland }

\email{trojkat@mat.ug.edu.pl}



\begin{abstract}
Let $N_{g,s}$ denote the nonorientable surface of genus $g$ with $s$ boundary components. Recently Paris and Szepietowski
\cite{SzepParis} obtained an explicit finite presentation for the mapping class group $\cM(N_{g,s})$ of the surface
$N_{g,s}$, where $s\in\{0,1\}$ and $g+s>3$. Following this work we obtain a finite presentation for the 
mapping class group $\cM(N_{g,s})$ with generators being Dehn twists and one crosscap slide. 
\end{abstract}

\maketitle%
\section{Introduction}%
Let $N_{g,s}$ be a smooth, nonorientable, compact surface of genus $g$ with $s$ boundary components. If
$s$ is zero, then we omit it from the notation. If we do not want to emphasise the numbers $g,s$, we simply 
write $N$ for a surface $N_{g,s}$. Recall that $N_{g}$ is a connected sum of $g$ projective planes 
and $N_{g,s}$ is obtained from $N_g$ by removing $s$ open disks.

Let ${\textrm{Diff}}(N)$ be the group of all diffeomorphisms $\map{h}{N}{N}$ such that $h$ is the identity 
on each boundary component. By ${\cal{M}}(N)$ we denote the quotient group of ${\textrm{Diff}}(N)$ by
the subgroup consisting of maps isotopic to the identity, where we assume that isotopies are 
the identity on each boundary component. ${\cal{M}}(N)$ is called the \emph{mapping class group} of $N$. 

The mapping class group ${\cal{M}}(S_{g,s})$ of an orientable surface is defined analogously, but we consider only
orientation preserving maps. 
\subsection{Background}
The problem of finding (decent) presentations for various mapping class groups has a long history. 
In the orientable case Birman and Hilden \cite{Bir-Hil} obtained a presentation for the so-called 
hyperelliptic mapping class group ${\cal{M}}^{h}(S_{g})$, hence in particular for the 
group ${\cal{M}}^{h}(S_{2})={\cal{M}}(S_{2})$. For $g\geq 3$ McCool \cite{McCool-pres} proved algebraically 
that the mapping class group ${\cal{M}}(S_{g})$ admits a finite presentation. The same result was proved
geometrically by Hatcher and Thurston \cite{Thur_Hat}. Later the Hatcher-Thurston approach was simplified by 
Harer \cite{Harer2} and used by Wajnryb \cite{Wajn_pre} to obtain a simple presentation for 
the groups ${\cal{M}}(S_{g})$ and ${\cal{M}}(S_{g,1})$. The Wajnryb presentation was used by various authors and 
led to some other interesting presentations -- see \cite{MatsumotoPres,Gervais_top,ParLab}. Let us emphasise 
that for surfaces without punctures, each of the above presentations use Dehn twists as generators. 

In the nonorientable case Lickorish \cite{Lick3} first observed that Dehn twists do not generate 
the mapping class group ${\cal{M}}(N_{g})$ for $g\geq 2$. More precisely, ${\cal{M}}(N_{g})$ is generated
by Dehn twists and one \emph{crosscap slide} (or $Y$-homeomorphism).
Later Chillingworth \cite{Chil} found a finite generating set for 
${\cal{M}}(N_{g})$. This generating set was extended to the case of a surface
with punctures and/or boundary components in \cite{Kork-non, Stukow_SurBg}.

As for presentations, Birman and Chillingworth \cite{BirChil1} found a simple presentation for the 
group ${\cal{M}}(N_{3})$. Bujalance, Costa and Gamboa derived algebraically a presentation for 
the hyperelliptic mapping class group \cite{Costa-hyper}. Later this presentation was obtained 
geometrically in \cite{PresHiperNon}. Szepietowski \cite{Szep_gen4}, following the ideas of 
Benvenuti \cite{Benvenuti}, was able to find a presentation for the group ${\cal{M}}(N_{4})$, and 
recently Paris and Szepietowski \cite{SzepParis} obtained a finite presentations for groups 
${\cal{M}}(N_{g,s})$, where $s\in\{0,1\}$ and $g+s>3$.
\subsection{Main results}
The presentations obtained by Paris and Szepietowski \cite{SzepParis} as generators use Dehn twists and $g-1$ 
crosscap transpositions (Theorems \ref{ParSzep1} and \ref{ParSzep2}). The main goal of this paper is to 
simplify these presentation by replacing all crosscap transpositions with one crosscap slide
(Theorems \ref{SimParSzep1} and \ref{SimParSzep3}). Our simplification not only leads to a simpler presentation 
(with fewer generators and relations), but also gives a more natural generating set (in fact closely related to 
the generating set found by Chillingworth \cite{Chil}). In the last section we give a simple geometric 
interpretation of obtained relations.

\section{Preliminaries}
\subsection{Notation}
Let us represent a surface $N_{g,0}$ and $N_{g,1}$ as respectively a sphere or a disc with $g$ crosscaps, and 
let $\alpha_1,\ldots,\alpha_{g-1},\beta$ be two-sided circles indicated in Figure~\ref{r01}. 
\begin{figure}[h]
\begin{center}
\includegraphics[width=0.95\textwidth]{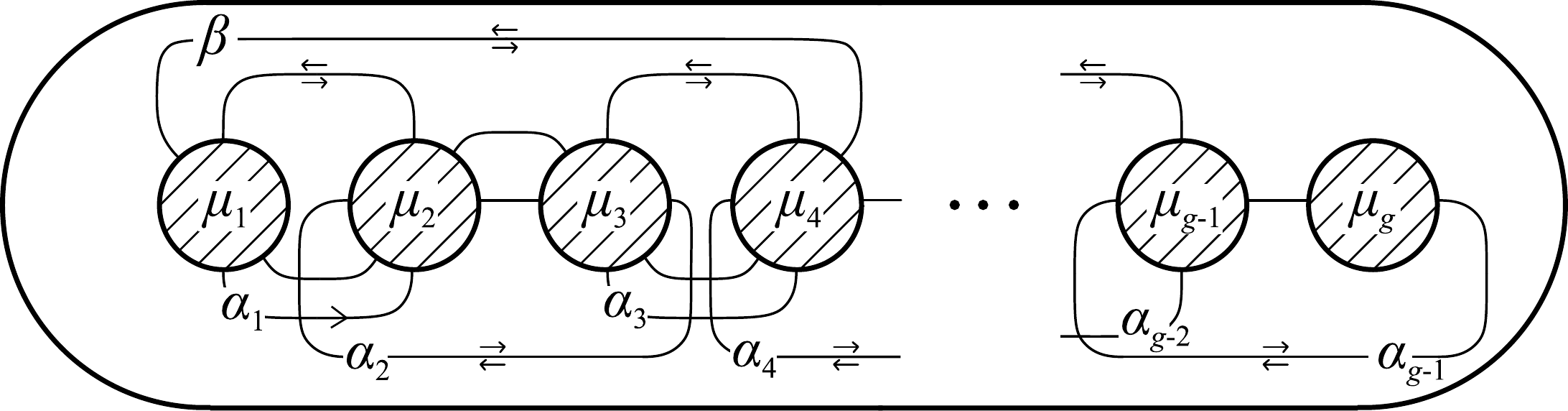}
\caption{Surface $N$ as a sphere/disc with crosscaps.}\label{r01} %
\end{center}
\end{figure}
Small arrows in this figure indicate directions of Dehn twists $a_1,\ldots,a_{g-1},b$ associated with these circles. 
Observe that $\beta$ (hence also $b$) is defined only if $g\geq 4$. From now on whenever we use $b$, we silently assume that
$g\geq 4$.

Moreover, for any two consecutive crosscaps $\mu_i,\mu_{i+1}$ we define a \emph{crosscap transposition} $u_i$ to be the map which 
interchanges these two crosscaps (see Figure~\ref{r02}).
\begin{figure}[h]
\begin{center}
\includegraphics[width=0.7\textwidth]{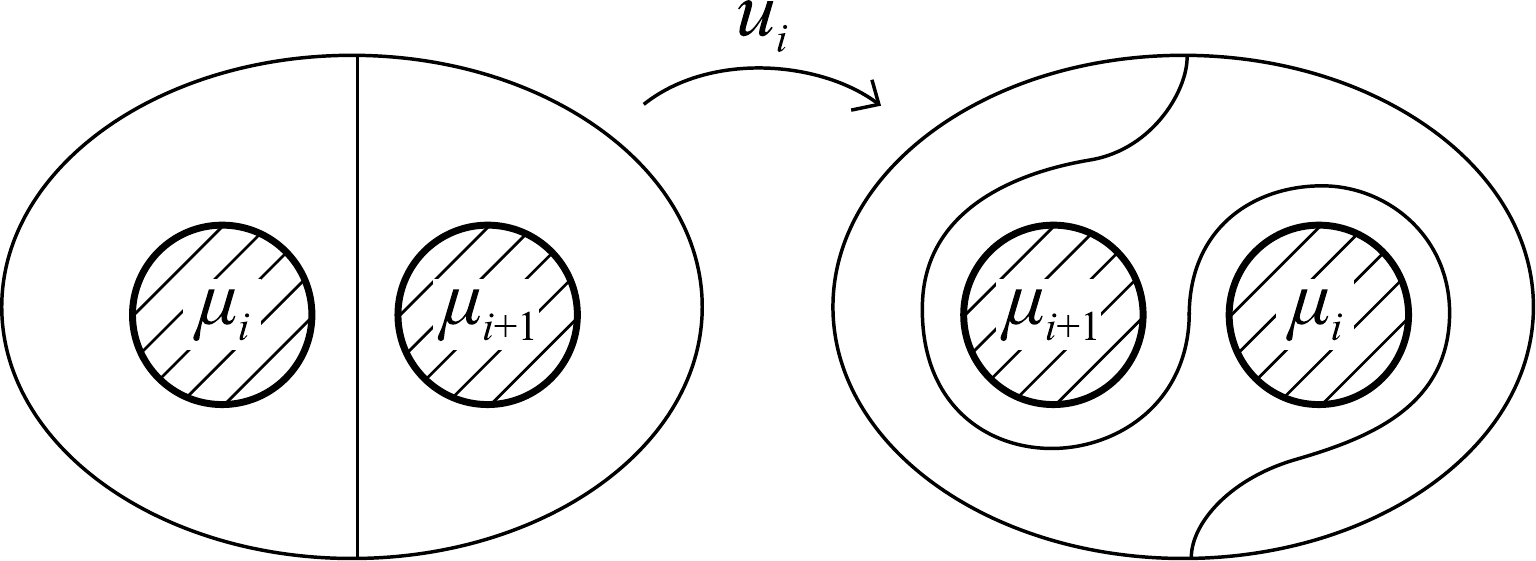}
\caption{Crosscap transposition.}\label{r02} %
\end{center}
\end{figure}
\subsection{The Paris-Szepietowski presentation for ${\cal{M}}(N)$}
The following two theorems are the main results of \cite{SzepParis}
\begin{tw}\label{ParSzep1}
If $g\geq 3$ is odd or $g=4$, then ${\cal{M}}(N_{g,1})$ admits a presentation with generators
$a_1,\ldots,a_{g-1},u_1,\ldots,u_{g-1}
$ and $b$ for $g\geq 4$. The defining relations are
\begin{enumerate}
 \item[(A1)] $a_ia_j=a_ja_i$\quad for $g\geq 4$, $|i-j|>1$,
\item[(A2)] $a_ia_{i+1}a_i=a_{i+1}a_ia_{i+1}$\quad for $i=1,\ldots,g-2$,
 \item[(A3)] $a_ib=ba_i$\quad for $g\geq 4,i\neq 4$,
 \item[(A4)] $ba_4b=a_4ba_4$\quad for $g\geq 5$,
 \item[(A5)] $(a_2a_3a_4b)^{10}=(a_1a_2a_3a_4b)^6$\quad for $g\geq 5$,
 \item[(A6)] $(a_2a_3a_4a_5a_6b)^{12}=(a_1a_2a_3a_4a_5a_6b)^{9}$\quad for $g\geq 7$,
\item[(B1)] $u_iu_j=u_ju_i$\quad for $g\geq 4$, $|i-j|>1$, 
\item[(B2)] $u_iu_{i+1}u_i=u_{i+1}u_iu_{i+1}$\quad for $i=1,\ldots,g-2$,
\item[(C1)] $a_1u_{i}=u_{i}a_1$\quad for $g\geq 4$, $i=3,\ldots,g-1$,
\item[(C2)] $a_iu_{i+1}u_i=u_{i+1}u_ia_{i+1}$\quad for $i=1,\ldots,g-2$,
\item[(C3)] $a_{i+1}u_iu_{i+1}=u_iu_{i+1}a_i$\quad for $i=1,\ldots,g-2$,
\item[(C4)] $a_1u_1a_1=u_1$,
\item[(C5)] $u_2a_1a_2u_1=a_1a_2$,
\item[(C6)] $(u_3b)^2=(a_1a_2a_3)^2(u_1u_2u_3)^2$\quad for $g\geq 4$,
\item[(C7)] $u_5b=bu_5$\quad for $g\geq 6$,
\item[(C8)] $ba_4u_4=a_4u_4(a_4a_3a_2a_1u_1u_2u_3u_4)b$\quad for $g\geq 5$.
\end{enumerate}
If $g\geq 6$ is even, then ${\cal{M}}(N_{g,1})$ admits a presentation with generators
$a_1,\ldots,a_{g-1}$, $u_1,\ldots,u_{g-1},b$ and additionally $b_0,b_1,\ldots,b_{\frac{g-2}{2}}$. The defining relations 
are relations (A1)--(A6), (B1)--(B2), (C1)--(C8) above and additionally
\begin{enumerate}
 \item[(A7)] $b_0=a_1, b_1=b$,
\item[(A8)] $b_{i+1}=(b_{i-1}a_{2i}a_{2i+1}a_{2i+2}a_{2i+3}b_{i})^5(b_{i-1}a_{2i}a_{2i+1}a_{2i+2}a_{2i+3})^{-6}$\quad\\	  for
$1\leq i\leq \frac{g-4}{2}$,
\item[(A9a)] $b_2b=bb_2$\quad for $g=6$,
\item[(A9b)] $b_{\frac{g-2}{2}}a_{g-5}=a_{g-5}b_{\frac{g-2}{2}}$\quad for $g\geq 8$. \qed
\end{enumerate}
\end{tw}
For the geometric interpretation of generators $b_0,b_1,\ldots,b_{\frac{g-2}{2}}$ and relations 
(A8), (A9) see
\cite{SzepParis}.
\begin{tw}\label{ParSzep2}
 If $g\geq 4$, then the group ${\cal{M}}(N_{g,0})$ is isomorphic to the quotient of the group ${\cal{M}}(N_{g,1})$
with presentation given in Theorem \ref{ParSzep1} by the relations
\begin{enumerate}
 \item[(B3)] $(u_{1}u_2\cdots u_{g-1})^g=1$,
\item[(B4)] $(u_{g-1}\cdots u_2u_1)(u_1u_2\cdots u_{g-1})=1$,
\item[(D)] $a_1(a_2\cdots a_{g-1}u_{g-1}\cdots u_2)a_1=a_2\cdots a_{g-1}u_{g-1}\cdots u_2$.\qed
\end{enumerate}
\end{tw}
\section{Simplified presentations for ${\cal{M}}(N_{g,1})$ and ${\cal{M}}(N_{g,0})$}
The goal of this section is to simplify the presentations of Theorems \ref{ParSzep1} and \ref{ParSzep2}. We will achieve this
by removing generators $u_2,\ldots,u_{g-1}$ and replacing $u_1$ by $y=a_1^{-1}u_1$.

Geometrically, $y$ is a \emph{crosscap slide} (or Y-homeomorphism), that is it is the effect of pushing the crosscap $\mu_1$ along
the curve $\alpha_1$ indicated in Figure \ref{r01}. For more details about the action of $y$ see Section \ref{sec:geom} below 
or Section 2.2 of \cite{SzepParis}. Note that Paris and Szepietowski use slightly different convention and define crosscaps in the form $a_1u_1$ (that is they push $\mu_2$ along $\alpha_1$), but we do not rely on any computations involving their notation for crosscaps, hence this leads to no ambiguities. 
\begin{tw} \label{SimParSzep1}
 If $g\geq 3$ is odd or $g=4$, then ${\cal{M}}(N_{g,1})$ admits a presentation with generators $a_1,\ldots,a_{g-1},y$ and $b$
for
$g\geq 4$. The defining relations are (A1)--(A6) and 
\begin{enumerate} 
 \item[(B1')] $y(a_2a_3a_1a_2ya_2^{-1}a_1^{-1}a_3^{-1}a_2^{-1})=(a_2a_3a_1a_2ya_2^{-1}a_1^{-1}a_3^{-1}a_2^{-1})y$
\quad for $g\geq 4$,
 \item[(B2')]  $y(a_2a_1y^{-1}a_2^{-1}ya_1a_2)y=a_1(a_2a_1y^{-1}a_2^{-1}ya_1a_2)a_1$,
\item[(C1')] $a_iy=ya_i$ for $g\geq 4$, $i=3,4,\ldots,g-1$,
\item[(C2')] $a_2(ya_2y^{-1})=(ya_2y^{-1})a_2$,
\item[(C4')] $a_1y=ya_1^{-1}$,
\item[(C6')] $byby^{-1}=[a_1a_2a_3(y^{-1}a_2y)a_3^{-1}a_2^{-1}a_1^{-1}][a_2^{-1}a_3^{-1}(ya_2y^{-1})a_3a_2]$
\quad for $g\geq 4$,
\item[(C7')] $(a_4a_5a_3a_4a_2a_3a_1a_2ya_2^{-1}a_1^{-1}a_3^{-1}a_2^{-1}a_4^{-1}a_3^{-1}a_5^{-1}a_4^{-1})b=\\
 b(a_4a_5a_3a_4a_2a_3a_1a_2ya_2^{-1}a_1^{-1}a_3^{-1}a_2^{-1}a_4^{-1}a_3^{-1}a_5^{-1}a_4^{-1})$
\quad for $g\geq 6$,
\item[(C8')]
$[(ya_1^{-1}a_2^{-1}a_3^{-1}a_4^{-1})b(a_4a_3a_2a_1y^{-1})][(a_1^{-1}a_2^{-1}a_3^{-1}a_4^{-1})b^{-1}(a_4a_3a_2a_1)]=\\
\ [(a_4^{-1 } a_3^{-1}a_2^{-1})y(a_2a_3a_4)][a_3^{-1}a_2^{-1}y^{-1}a_2a_3]
[a_2^{-1}ya_2]y^{-1}$ \quad for $g\geq~5$.
\end{enumerate}
If $g\geq 6$ is even, then ${\cal{M}}(N_{g,1})$ admits a presentation with generators
$a_1,\ldots,a_{g-1}$, $y$, $b$ and additionally $b_0,b_1,\ldots,b_{\frac{g-2}{2}}$. The defining relations 
are relations (A1)--(A9), (B1')--(B2') and (C1')--(C8').
\end{tw}
\begin{proof}
 Let us begin by adding to the presentation of Theorem \ref{ParSzep1} two generators $y,M$ and five relations
\begin{enumerate}
 \item[(X1)] $u_1=a_1y$,
\item[(X2)] $M=a_1a_2\cdots a_{g-1}$,
\item[(X3)] $Mu_i=u_{i+1}^{-1}M$\quad for $i=1,\ldots,g-2$,
\item[(X4)] $Ma_i=a_{i+1}M$\quad for $i=1,\ldots,g-2$,
\item[(C1')] $a_iy=ya_i$ for $g\geq 4$, $i=3,4,\ldots,g-1$.
\end{enumerate}
Relations (X1) and (X2) simply define $y$ and $M$ in terms of the remaining generators. 
Relations (X3) and (C1') are easy consequences of relations 
\begin{enumerate}
 \item[(C1a)] $a_iu_j=u_ja_i$\quad for $g\geq 4$, $|i-j|>1$,
\item[(C5a)] $a_ia_{i+1}u_i=u_{i+1}^{-1}a_ia_{i+1}$\quad for $i=1,\ldots,g-2$,
\end{enumerate}
which are proved in Lemma 3.8 of \cite{SzepParis}. Relation (X4) is an easy consequence of (A1) and (A2).

Observe that relation (X3) allows to inductively remove
$u_2,u_3,\ldots,u_{g-1}$ from the presentation, and relation (X1) allows to replace $u_1$ with $y$. 
In order to complete this
task we need to rewrite relations involving $u_1,\ldots,u_{g-1}$, that is relations (B1)--(B2) and 
(C1)--(C8).\\
{\bf (C4)} It is straightforward to check that relation (C4) can be rewritten as
\begin{enumerate}
 \item[(C4')] $a_1y=ya_1^{-1}$.
\end{enumerate}
Using this relation and relations (C1'), (X1), (X2), (A1), (A2) we can rewrite (X3) for $i\leq 5$ as follows
\[\begin{aligned}
      u_2&=Mu_1^{-1}M^{-1}=(a_1\cdots a_{g-1})(a_1y^{-1})(a_{g-1}^{-1}\cdots a_1^{-1})=a_1a_2(a_1y^{-1})a_2^{-1}a_1^{-1},\\
   u_3&=Mu_2^{-1}M^{-1}=(a_1\cdots a_{g-1})a_1a_2(a_1y)a_2^{-1}a_1^{-1}(a_{g-1}^{-1}\cdots a_1^{-1})=\\
   &=a_1a_2a_3\SL{a_1a_2}(a_1y)\SR{a_2^{-1}a_1^{-1}}a_3^{-1}a_2^{-1}a_1^{-1}
=\\
&=(a_2a_3)a_1a_2\ukre{a_3}(a_1y)\ukre{a_3^{-1}}a_2^{-1}a_1^{-1}(a_3^{-1}a_2^{-1})=a_2a_3a_1a_2(a_1y)a_2^{-1}a_1^{-1}a_3^{-1}a_2^{-1},\\
   u_4&=Mu_3^{-1}M^{-1}=(a_1\cdots a_{g-1})a_2a_3a_1a_2(a_1y^{-1})a_2^{-1}a_1^{-1}a_3^{-1}a_2^{-1}(a_{g-1}^{-1}\cdots
a_1^{-1})=\\
   &=(a_1a_2a_3a_4)\SL{a_2a_3a_1a_2}(a_1y^{-1})\SR{a_2^{-1}a_1^{-1}a_3^{-1}a_2^{-1}}(a_4^{-1}a_3^{-1}a_2^{-1}a_1^{-1})=\\
   &=(a_3a_4a_2a_3)a_1a_2\ukre{a_3a_4}(a_1y^{-1})\ukre{a_4^{-1}a_3^{-1}}a_2^{-1}a_1^{-1}(a_3^{-1}a_2^{-1}a_4^{-1}a_3^{-1})=\\
   &=a_3a_4a_2a_3a_1a_2(a_1y^{-1})a_2^{-1}a_1^{-1}a_3^{-1}a_2^{-1}a_4^{-1}a_3^{-1},\\
   u_5&=Mu_4^{-1}M^{-1}=\\
   &=(a_1\cdots
a_{g-1})\SL{a_3a_4a_2a_3a_1a_2}(a_1y)\SR{a_2^{-1}a_1^{-1}a_3^{-1}a_2^{-1}a_4^{-1}a_3^{-1}}(a_{g-1}^{-1}\cdots a_1^{-1})=\\
   &=a_4a_5a_3a_4a_2a_3a_1a_2(a_1y)a_2^{-1}a_1^{-1}a_3^{-1}a_2^{-1}a_4^{-1}a_3^{-1}a_5^{-1}a_4^{-1}.
  \end{aligned}
\]
In particular relation {\bf (C5)} is superfluous. 

In the above computations we introduced the notation which should help the reader to follow our transformations. The underlined parts indicate expressions which will be reduced, and parts with small arrows indicate expressions which will be moved to the left/right.

Let us also explain that by 'rewriting a relation $R$' we mean transforming it into a relation which equivalent to $R$ under some, previously specified, set of relations. \\
{\bf (B1)} In order to rewrite relation (B1) assume that $j=i+k$, where $k\geq 2$. By (X3) we have 
$u_i=M^{(i-1)}u_1^{(-1)^{i-1}}M^{-(i-1)}$ and 
\[u_j=u_{i+k}=M^{(i+k-1)}u_1^{(-1)^{i+k-1}}M^{-(i+k-1)},\] hence (B1) takes form
\[\begin{aligned}
   &M^{(i-1)}u_1^{(-1)^{i-1}}M^{-(i-1)}\cdot
M^{(i+k-1)}u_1^{(-1)^{i+k-1}}M^{-(i+k-1)}=\\
&\quad\qquad\qquad\qquad=M^{(i+k-1)}u_1^{(-1)^{i+k-1}}M^{-(i+k-1)} \cdot
M^{(i-1)}u_1^{(-1)^{i-1}}M^{-(i-1)},\\
   &u_1^{(-1)^{i-1}}\left(M^{k}u_1^{(-1)^{i+k-1}}M^{-k}\right)=\left(M^{k}u_1^{(-1)^{i+k-1}}M^{-k}\right)u_1^{(-1)^{i-1}}.
  \end{aligned}
\]
What we get is the commutativity relation between  $u_1^{(-1)^{i-1}}$ and $M^{k}u_1^{(-1)^{i+k-1}}M^{-k}$,
hence this relation is equivalent to the commutativity relation between $u_1$ and $M^ku_1M^{-k}$.
\[u_1M^{k}u_1M^{-k}=M^{k}u_1M^{-k}u_1.\]
We claim that this relation can be reduced to the case $k=2$. If $k>2$, by (A1), (A2), (X1), (X2), (X4), (C1') 
we have
\[\begin{aligned}   
   &u_1M^{k-1}\SL{a_1a_2}\ukre{a_3\cdots a_{g-1}}u_1\ukre{a_{g-1}^{-1}\cdots a_3^{-1}}\SR{a_2^{-1}a_1^{-1}}M^{-(k-1)}=\\
   &\qquad\qquad\qquad=M^{k-1}\SL{a_1a_2}\ukre{a_3\cdots a_{g-1}}u_1\ukre{a_{g-1}^{-1}\cdots a_3^{-1}}\SR{a_2^{-1}a_1^{-1}}M^{-(k-1)}u_1,\\
   &\ukre{a_{k}a_{k+1}}u_1M^{k-1}u_1M^{-(k-1)}\ukre{a_{k+1}^{-1}a_k^{-1}}=
   \ukre{a_{k}a_{k+1}}M^{k-1}u_1M^{-(k-1)}u_1\ukre{a_{k+1}^{-1}a_k^{-1}},\\
   &u_1M^{k-1}u_1M^{-(k-1)}=M^{k-1}u_1M^{-(k-1)}u_1,\\
   &\cdots\\
   &u_1M^{2}u_1M^{-2}=M^{2}u_1M^{-2}u_1,\\
   &u_1u_3=u_3u_1.
  \end{aligned}
\]
Now we substitute for $u_1$ and $u_3$.
\[\begin{aligned}  
&a_1y(a_2a_3a_1a_2(\SL{a_1}y)a_2^{-1}a_1^{-1}a_3^{-1}a_2^{-1})=(a_2a_3a_1a_2(\SL{a_1}y)a_2^{-1}a_1^{-1}a_3^{-1}a_2^{-1})\SL{
a_1}y,\\  
&\ukre{a_3a_1}y(a_2a_3a_1a_2ya_2^{-1}a_1^{-1}a_3^{-1}a_2^{-1})=\ukre{a_3a_1}(a_2a_3a_1a_2ya_2^{-1}a_1^{-1}a_3^{-1}a_2^{-1}
)y,\\
   &y(a_2a_3a_1a_2ya_2^{-1}a_1^{-1}a_3^{-1}a_2^{-1})=(a_2a_3a_1a_2ya_2^{-1}a_1^{-1}a_3^{-1}a_2^{-1})y.
  \end{aligned}
\]
This is exactly relation (B1').\\
{\bf (B2)}
By (X3), relation (B2) takes form
\[\begin{aligned}
   &(M^{(i-1)}u_1^{(-1)^{i-1}}M^{-(i-1)})(M^{i}u_1^{(-1)^{i}}M^{-i})(M^{(i-1)}u_1^{(-1)^{i-1}}M^{-(i-1)})=\\
   &\qquad\qquad\qquad\qquad =(M^{i}u_1^{(-1)^{i}}M^{-i})(M^{(i-1)}u_1^{(-1)^{i-1}}M^{-(i-1)})(M^{i}u_1^{(-1)^{i}}M^{-i}),\\
   &u_1^{(-1)^{i-1}}Mu_1^{(-1)^{i}}M^{-1}u_1^{(-1)^{i-1}}=
   Mu_1^{(-1)^{i}}M^{-1}u_1^{(-1)^{i-1}}Mu_1^{(-1)^{i}}M^{-1}.
  \end{aligned}
\]
Observe that we can assume (possibly taking the inverse of the relation) that $u_1^{(-1)^i}=u_1^{-1}$, and the
relation takes form
\[\begin{aligned}
   &u_1Mu_1^{-1}M^{-1}u_1=Mu_1^{-1}M^{-1}u_1Mu_1^{-1}M^{-1},\\
   &u_1u_2u_1=u_2u_1u_2,\\  
&\ukre{a_1}y(a_1a_2a_1y^{-1}a_2^{-1}\ukre{a_1^{-1}})\ukre{a_1}y=(\ukre{a_1}a_2a_1y^{-1}a_2^{-1}\ukre{a_1^{-1}})\ukre{a_1}
y(a_1a_2a_1\SR{y^{-1}a_2^{-1}a_1^{-1}}),\\
   &y\SL{a_1}a_2a_1y^{-1}a_2^{-1}ya_1a_2y=a_2a_1y^{-1}a_2^{-1}ya_1a_2a_1,\\
  &y(a_2a_1y^{-1}a_2^{-1}ya_1a_2)y=a_1(a_2a_1y^{-1}a_2^{-1}ya_1a_2)a_1.
  \end{aligned}
\]
Hence (B2) is equivalent to (B2').\\
{\bf (C1)} 
Let $k=i-1\geq 2$. By (X3), relation (C1) takes form
\[a_1M^{k}u_1^{(-1)^{k}}M^{-k}=M^{k}u_1^{(-1)^{k}}M^{-k}a_1.\]
Since this is a commutativity relation, we can assume that $u_1^{(-1)^{k}}=u_1$, hence we have
\[a_1M^{k}u_1M^{-k}=M^{k}u_1M^{-k}a_1.\]
We claim that this relation can be reduced to the case $k=2$. If $k>2$ by relations 
(A1), (A2), (X1)--(X3), (C1'), we have
\[\begin{aligned}   
   &a_1M^{k-1}\SL{a_1a_2}\ukre{a_3\cdots a_{g-1}}u_1\ukre{a_{g-1}^{-1}\cdots a_3^{-1}}\SR{a_2^{-1}a_1^{-1}}M^{-(k-1)}=\\
   &\qquad\qquad\qquad=M^{k-1}\SL{a_1a_2}\ukre{a_3\cdots a_{g-1}}u_1\ukre{a_{g-1}^{-1}\cdots a_3^{-1}}\SR{a_2^{-1}a_1^{-1}}M^{-(k-1)}a_1,\\
   &\ukre{a_ka_{k+1}}a_1M^{k-1}u_1M^{-(k-1)}\ukre{a_{k+1}^{-1}a_k^{-1}}=
   \ukre{a_ka_{k+1}}M^{k-1}u_1M^{-(k-1)}a_1\ukre{a_{k+1}^{-1}a_k^{-1}},\\
   &a_1M^{k-1}u_1M^{-(k-1)}=M^{k-1}u_1M^{-(k-1)}a_1,\\
   &\ldots\\
   &a_1M^{2}u_1M^{-2}=M^{2}u_1M^{-2}a_1,\\
   &a_1u_3=u_3a_1.
\end{aligned}
\]
Now we substitute for $u_3$.
\[\begin{aligned}    
&a_1(a_2[a_3a_1]a_2(a_1y)a_2^{-1}[a_1^{-1}a_3^{-1}]a_2^{-1})=(a_2[a_3a_1]a_2(a_1y)a_2^{-1}[a_1^{-1}a_3^{-1}]a_2^{-1})a_1,\\
   &\SR{a_1}(a_2a_1a_3a_2a_1ya_2^{-1}a_3^{-1}a_1^{-1}a_2^{-1})=(a_2a_1a_3a_2a_1ya_2^{-1}a_3^{-1}a_1^{-1}a_2^{-1})\SL{a_1},\\
 &\ukre{a_2a_1a_3a_2a_1}(a_3y)\ukre{a_2^{-1}a_3^{-1}a_1^{-1}a_2^{-1}}=\ukre{a_2a_1a_3a_2a_1}(ya_3)\ukre{a_2^{-1}a_3^{-1}
a_1^{-1}a_2^{-1}},\\   
   &a_3y=ya_3.
  \end{aligned}
\]
Therefore (C1) is a special case of (C1').\\
{\bf (C2)} We rewrite (C2) using relations (A2), (X1)--(X4), (C1') and (C4'):
\[\begin{aligned}
   &a_i(M^{i}u_1^{(-1)^{i}}M^{-i})(M^{(i-1)}u_1^{(-1)^{i-1}}M^{-(i-1)})=\\
   &\qquad\qquad\qquad\qquad\qquad=(M^{i}u_1^{(-1)^{i}}M^{-i})(M^{(i-1)}u_1^{(-1)^{i-1}}M^{-(i-1)})\SL{a_{i+1}},\\
   &\SR{a_i}M^{i-1}(Mu_1^{(-1)^{i}}M^{-1})(u_1^{(-1)^{i-1}}\ukre{M^{-(i-1)}})=\\
   &\qquad\qquad\qquad\qquad\qquad\qquad\qquad=(M^iu_1^{(-1)^{i}}M^{-1})(u_1^{(-1)^{i-1}}a_2\ukre{M^{-(i-1)}}),\\
   &\ukre{M^{i-1}}a_1Mu_1^{(-1)^{i}}M^{-1}u_1^{(-1)^{i-1}}=\ukre{M^{i-1}} Mu_1^{(-1)^{i}}M^{-1}u_1^{(-1)^{i-1}}a_2,\\
   &\ukre{a_1}[a_1a_2a_1]y^{(-1)^{i}}a_2^{-1}\ukre{a_1^{-1}a_1}y^{(-1)^{i-1}}= \ukre{a_1}a_2a_1y^{(-1)^{i}}a_2^{-1}\ukre{a_1^{-1}a_1}y^{(-1)^{i-1}}a_2,\\
   &\ukre{a_2a_1}a_2y^{(-1)^{i}}a_2^{-1}y^{(-1)^{i-1}}= \ukre{a_2a_1}y^{(-1)^{i}}a_2^{-1}y^{(-1)^{i-1}}a_2,\\
   &a_2\left(y^{(-1)^{i}}a_2^{-1}y^{(-1)^{i-1}}\right)= \left(y^{(-1)^{i}}a_2^{-1}y^{(-1)^{i-1}}\right)a_2.
  \end{aligned}
\]
Obtained relation is equivalent to (C2').\\
{\bf (C3)} We rewrite (C3) using relations (A2), (X1)--(X4), (C1') and (C4'):
\[\begin{aligned}
   &\SR{a_{i+1}}(M^{(i-1)}u_1^{(-1)^{i-1}}M^{-(i-1)})(M^{i}u_1^{(-1)^{i}}M^{-i})=\\
&\qquad\qquad\qquad\qquad\qquad\qquad=(M^{(i-1)}u_1^{(-1)^{i-1}}M^{-(i-1)})(M^{i}u_1^{(-1)^{i}}M^{-i})a_i,\\
   &(\ukre{M^{(i-1)}}a_2u_1^{(-1)^{i-1}})(Mu_1^{(-1)^{i}}M^{-i})=\\
 &\qquad\qquad\qquad\qquad\qquad\qquad=(\ukre{M^{(i-1)}}u_1^{(-1)^{i-1}})(Mu_1^{(-1)^{i}}M^{-1}M^{-(i-1)})\SL{a_i},\\
   &a_2u_1^{(-1)^{i-1}}Mu_1^{(-1)^{i}}M^{-1}\ukre{M^{-(i-1)}}=u_1^{(-1)^{i-1}}Mu_1^{(-1)^{i}}M^{-1}a_1\ukre{M^{-(i-1)}},\\
   &a_2\ukre{a_1}y^{(-1)^{i-1}}\ukre{a_1}a_2\SR{a_1}y^{(-1)^{i}}a_2^{-1}a_1^{-1}=\ukre{a_1}y^{(-1)^{i-1}}\ukre{a_1}a_2\SR{a_1}y^{(-1)^{i}}a_2^{-1}\ukre{a_1^{-1}a_1},\\
   &a_2y^{(-1)^{i-1}}a_2y^{(-1)^{i}}a_2^{-1}\ukre{a_1^{-1}a_2^{-1}}=y^{(-1)^{i-1}}a_2y^{(-1)^{i}}\ukre{a_1^{-1}a_2^{-1}},\\
   &a_2y^{(-1)^{i-1}}a_2y^{(-1)^{i}}a_2^{-1}=y^{(-1)^{i-1}}a_2y^{(-1)^{i}}.
  \end{aligned} 
\]
Once again the obtained relation is equivalent to (C2').\\
{\bf (C6)} Using relations (A1), (A2), (X3), (C4'), we rewrite $u_1u_2u_3$ as
\[\begin{aligned}
   u_1u_2u_3&=\ukre{a_1}y(\ukre{a_1}a_2(a_1y^{-1})a_2^{-1}a_1^{-1})(\SL{a_2}a_3a_1a_2(a_1y)a_2^{-1}a_1^{-1}a_3^{-1}a_2^{-1})=\\
&=ya_2\ukre{a_1}y^{-1}\ukre{a_1}a_2^{-1}\ukre{a_1^{-1}}a_3\ukre{a_1}a_2a_1ya_2^{-1}a_1^{-1}a_3^{-1}a_2^{-1}=\\
&=ya_2y^{-1}a_2^{-1}a_3a_2a_1ya_2^{-1}a_1^{-1}a_3^{-1}a_2^{-1}.
  \end{aligned}
\]
Therefore, using relations (A1)--(A3), (X1)--(X3), (C1'), (C2'), (C4'), we can rewrite relation (C6) as
\[\begin{aligned}
  &(a_2a_3a_1a_2(\ukre{a_1}y)\ukre{a_2^{-1}a_1^{-1}a_3^{-1}a_2^{-1}})b(\ukre{a_2a_3a_1a_2}(\ukre{a_1}y)\ukre{a_2^{-1}a_1^{-1}a_3^{-1}a_2^{-1}})b=\\
  &\qquad\qquad\qquad =
(a_1a_2a_3)(\SL{a_1a_2}a_3)(ya_2y^{-1}a_2^{-1}a_3a_2a_1ya_2^{-1}a_1^{-1}a_3^{-1}a_2^{-1})\cdot\\
&\qquad\qquad\qquad\qquad\qquad\qquad\qquad\qquad\qquad\cdot (ya_2y^{-1}a_2^{-1}a_3a_2a_1y\ukre{a_2^{-1}a_1^{-1}a_3^{-1}a_2^{-1}}),\\
  &\ukre{a_2a_3a_1a_2}\ukre{y}byb
=\ukre{a_2a_3}(\ukre{a_1a_2}\SL{a_3})a_3(\ukre{y}a_2y^{-1}\SR{a_2^{-1}}a_3a_2\SR{a_1}ya_2^{-1}a_1^{-1}a_3^{-1}a_2^{-1})\cdot \\
&\qquad\qquad\qquad\qquad\qquad\qquad\qquad\qquad\qquad\qquad\qquad\quad\cdot(ya_2y^{-1}\SR{a_2^{-1}}a_3a_2a_1\SR{y}),\\
  &\ukre{a_3^{-1}}byby^{-1}\ukre{a_3}=
a_3a_2y^{-1}\SL{a_3}a_2\SR{a_3^{-1}}ya_2^{-1}\SR{a_1^{-1}}a_3^{-1}a_2^{-1}a_3^{-1}ya_2y^{-1}a_3a_2\SR{a_1},\\
  &byby^{-1}\SL{a_1^{-1}}=
a_2a_3[a_2y^{-1}a_2y]a_3^{-1}a_2^{-1}\SL{a_3^{-1}}a_1^{-1}[a_2^{-1}a_3^{-1}(ya_2y^{-1})a_3a_2],\\
  &\SL{a_1^{-1}}byby^{-1}=
a_2a_3y^{-1}a_2y\ukre{a_2a_2^{-1}}a_3^{-1}a_2^{-1}a_1^{-1}[a_2^{-1}a_3^{-1}(ya_2y^{-1})a_3a_2],\\
&byby^{-1}=
[a_1a_2a_3(y^{-1}a_2y)a_3^{-1}a_2^{-1}a_1^{-1}][a_2^{-1}a_3^{-1}(ya_2y^{-1})a_3a_2].
  \end{aligned}
\]
This is (C6').\\
{\bf (C7)} Using relations (A1)--(A3), (X1)--(X3), (C1'), (C4'),  we rewrite relation (C7) as
\[\begin{aligned}
   &(a_4a_5a_3a_4a_2a_3a_1a_2\SL{a_1}ya_2^{-1}a_1^{-1}a_3^{-1}a_2^{-1}a_4^{-1}a_3^{-1}a_5^{-1}a_4^{-1})b=\\
&\qquad\qquad\qquad\qquad =b(a_4a_5a_3a_4a_2a_3a_1a_2\SL{a_1}ya_2^{-1}a_1^{-1}a_3^{-1}a_2^{-1}a_4^{-1}a_3^{-1}a_5^{-1}a_4^{-1}),\\
   &\ukre{a_5}(a_4a_5a_3a_4a_2a_3a_1a_2ya_2^{-1}a_1^{-1}a_3^{-1}a_2^{-1}a_4^{-1}a_3^{-1}a_5^{-1}a_4^{-1})b=\\
&\qquad\qquad\qquad\qquad =\ukre{a_5}b(a_4a_5a_3a_4a_2a_3a_1a_2ya_2^{-1}a_1^{-1}a_3^{-1}a_2^{-1}a_4^{-1}a_3^{-1}a_5^{-1}a_4^{-1}).
  \end{aligned}
\]
This is exactly (C7').\\
{\bf (C8)} Using relations (A1)--(A3), (X1)--(X3), (C1'), (C2'), (C4'), we rewrite relation (C8) as
\[\begin{aligned}
&b\SR{a_4}(a_3a_4a_2a_3a_1a_2(a_1y^{-1})a_2^{-1}\ukre{a_1^{-1}}a_3^{-1}\ukre{a_2^{-1}}a_4^{-1}\ukre{a_3^{-1}})=\\
&\qquad\quad=\SR{a_4}(a_3a_4a_2a_3a_1a_2(a_1y^{-1})a_2^{-1}a_1^{-1}a_3^{-1}a_2^{-1}a_4^{-1}a_3^{-1})
\SL{a_4}a_3a_2a_1(\ukre{a_1}y)\cdot\\
&\qquad\qquad\qquad\qquad\quad \cdot (\ukre{a_1}a_2(\SR{a_1}y^{-1})a_2^{-1}a_1^{-1})
(a_2a_3a_1a_2(a_1y)a_2^{-1}a_1^{-1}a_3^{-1}a_2^{-1})\cdot\\
&\qquad\qquad\qquad\qquad\qquad\qquad\quad
\cdot(\SL{a_3}a_4a_2a_3a_1a_2(a_1y^{-1})a_2^{-1}\ukre{a_1^{-1}}a_3^{-1}\ukre{a_2^{-1}}a_4^{-1}\ukre{a_3^{-1}})b,\\
&b(\ukre{a_3}a_4\ukre{a_2}a_3\ukre{a_1}a_2a_1(\SR{a_1}y^{-1})a_2^{-1}a_3^{-1}a_4^{-1})=\\
&\qquad\quad=(\ukre{a_3}\SL{a_4}\ukre{a_2}\SL{a_3}\ukre{a_1}\SL{a_2a_1}(\ukre{a_1}\SL{y^{-1}})\ukre{a_1}a_2^{-1}\ukre{a_1^{-1}}a_3^
{-1}\ukre{a_2^ { -1 } } a_4^ { -1}\ukre{a_3^{-1}})\ukre{a_3a_2a_1}(y)\cdot\\
&\qquad\qquad\qquad\qquad\quad \cdot
(a_2(y^{-1})a_2^{-1}\ukre{a_1^{-1}a_2^{-1}})(\ukre{a_2}a_3\ukre{a_1}a_2(\ukre{a_1}y)\ukre{a_1}a_2^{-1}
\ukre{a_1^{-1}}a_3^{-1}\ukre{a_2^{-1}})\cdot\\
&\qquad\qquad\qquad\qquad\qquad\qquad\qquad\qquad\qquad\quad \cdot
a_4\ukre{a_2}a_3\ukre{a_1}a_2(\SR{a_1}y^{-1})\SR{a_2^{-1}a_3^{-1}a_4^{-1}b},\\  
&[(ya_1^{-1}a_2^{-1}a_3^{-1}a_4^{-1})b(a_4a_3a_2a_1y^{-1})][(a_1^{-1}a_2^{-1}a_3^{-1}a_4^{-1})b^{-1}
(a_4a_3a_2a_1)]=\\
&\qquad\qquad\qquad\qquad\quad\qquad=a_2^{-1}a_3^{-1}a_4^{-1}
ya_2y^{-1}\SL{a_2^{-1}}a_3a_2\SR{y}\SR{a_2^{-1}a_3^{-1}}a_4a_3a_2y^{-1}.
  \end{aligned}
\]
The left-hand side of the obtained relation is the same as the left-hand side of (C8'), so now let us concentrate 
on the right-hand side. Observe also that the left-hand side commutes with $a_3$ and $a_4$ -- we will use this fact below.
\[\begin{aligned}
   &a_3^{-1}a_2^{-1}a_3^{-1}\SR{a_4^{-1}}ya_2y^{-1}a_3a_2\SL{a_4}a_3ya_2y^{-1}\SR{a_3^{-1}a_4^{-1}},\\
&\SR{a_3^{-1}}a_4^{-1}a_3^{-1}a_2^{-1}a_3^{-1}\SL{y}a_2\SR{y^{-1}}\SR{a_4^{-1}}a_3a_4a_2a_3ya_2y^{-1},\\
&(a_4^{-1}a_3^{-1}a_2^{-1})y\SR{a_4^{-1}a_3^{-1}}a_2a_3a_4y^{-1}a_2a_3[a_2^{-1}ya_2]y^{-1},\\
&[(a_4^{-1}a_3^{-1}a_2^{-1})y(a_2a_3a_4)][a_3^{-1}a_2^{-1}y^{-1}a_2a_3][a_2^{-1}ya_2]y^{-1}.
  \end{aligned}
\]
This is exactly the right-hand side of (C8').

In order to finish the proof we remove generators $u_1,u_2,\ldots,u_{g-1},M$ from the presentation together with their
defining relations: (X1), (X3) and (X2).
\end{proof}
\begin{wn}
The presentation given in Theorem \ref{ParSzep1} is equivalent to the presentation in which 
we remove generators $u_6,u_7,\ldots,u_{g-1}$ and replace relations (B1), (B2), (C1)--(C5) with relations
\begin{itemize}
 \item[(B1)] $u_1u_3=u_3u_1$\quad for $g\geq 4$,
\item[(B2)] $u_1u_2u_1=u_2u_1u_2$,
\item[(C1')] $u_1a_i=a_iu_1$\quad for $g\geq 4$, $i=3,\ldots,g-1$,
\item[(C2)] $a_1u_2u_1=u_2u_1a_2$,
\item[(C4)] $a_1u_1a_1=u_1$,
\item[(C5a)] $u_{i+1}a_ia_{i+1}=a_ia_{i+1}u_i^{-1}$\quad for $i=1,2,3,4$.
\end{itemize}
\end{wn}
\begin{proof}
 If we remove from the described presentation generators $u_2,u_3,u_4,u_5$ (using relation (C5a)) and 
replace $u_1$ with $y=a_1^{-1}u_1$, we obtain the presentation given in Theorem \ref{SimParSzep1}.
\end{proof}
Now we turn to the case of a closed surface.
\begin{prop} \label{SimParSzep2}
 If $g\geq 4$, then the group ${\cal M}(N_{g,0})$ is isomorphic to the quotient of the group ${\cal M}(N_{g,1})$ with
presentation given in Theorem \ref{SimParSzep1} obtained by adding a generator $\ro$ and relations
\begin{itemize}
 \item[(B3')] $(a_1a_2\cdots a_{g-1})^g=
 \begin{cases}
1&\text{for $g$ even}\\
\ro&\text{for $g$ odd,}
\end{cases}$
\item[(D')] $\ro a_1=a_1\ro$,
\item[(E)] $\ro^2=1$,
\item[(Fa)] $(y^{-1}a_2a_3\cdots a_{g-1}ya_2a_3\cdots a_{g-1})^{\frac{g-1}{2}}=1$\quad for $g$ odd,
\item[(Fb)] $(y^{-1}a_2a_3\cdots a_{g-1}ya_2a_3\cdots a_{g-1})^{\frac{g-2}{2}}y^{-1}a_2a_3\cdots a_{g-1}=\ro$\quad for $g$ even.
\end{itemize}
\end{prop}
\begin{proof}
 We start from the presentation given by Theorem \ref{ParSzep2} and as in the case of a surface with 
boundary, we eliminate generators $u_1,\ldots,u_{g-1}$ using relations (X1)--(X4). 
However, before we do that, we add a generator $\ro$ with the defining relation
\begin{itemize}
 \item[(F)] $\ro=a_1a_2\cdots a_{g-1}u_{g-1}\cdots u_2u_1$.
\end{itemize}
The element $\ro$ represents the \emph{hyperelliptic involution}, that is the reflection across the 
plane containing centers of crosscaps in Figure \ref{r01} -- for details see \cite{PresHiperNon}.

We also add the relation (E) which is a consequence of relations (A1), (A2), (B1)--(B3), (C1)--(C5),
(F) -- for details see Lemma 3.9 of \cite{SzepParis} ($\ro$ is denoted as $r_g$ in that lemma).

By relations (C4) and (F), relation {\bf (D)} is equivalent to (D'). Using relations (A1), (A2), (B1), (B2), (C1)--(C5),
(D'), (F) one can
prove that 
\[\begin{aligned}
   &\ro a_i \ro^{-1}=a_i\quad\text{for $i=1,\ldots,g-1$,}\\
&\ro u_i \ro^{-1}=u_i^{-1}\quad\text{for $i=1,\ldots,g-1$.}
  \end{aligned}
\]
For details see Lemmas 3.9 and 3.10 of \cite{SzepParis}. Using 
these relations let us conjugate relation (F) by $\ro$.
\[\ro=\ro(a_1a_2\cdots a_{g-1}u_{g-1}\cdots u_2u_1)\ro^{-1}=
a_1a_2\cdots a_{g-1}u_{g-1}^{-1}\cdots u_2^{-1}u_1^{-1}.\]
Using this last relation and (F) we can prove that relation {\bf (B4)} is superfluous.
\[\begin{aligned}
&(u_1u_2\cdots u_{g-1})(u_{g-1}\cdots u_2u_1)=1,\\
   &(u_1u_2\cdots u_{g-1})(a_{g-1}^{-1}\cdots a_2^{-1}a_1^{-1})(a_1a_2\cdots a_{g-1})(u_{g-1}\cdots u_2u_1)=1,\\
&\ro^{-1}\ro =1.
  \end{aligned}
\]
{\bf (B3)} Now we rewrite (B3).
\[\begin{aligned}
   &(\ro^{-1} a_1a_2\cdots a_{g-1})^g=1,\\
&\ro^{-g}(a_1a_2\cdots a_{g-1})^g=1.
  \end{aligned}
\]
Since $\ro^2=1$, this relation is equivalent to (B3').\\
{\bf (F)} Finally, we need to substitute for $u_1,\ldots,u_{g-1}$ in relation (F).
\[\begin{aligned}
\ro&=a_1a_2\cdots a_{g-1}\cdot 
(M^{g-2}u_1^{(-1)^{g-2}}M^{-(g-2)})(M^{g-3}u_1^{(-1)^{g-3}}M^{-(g-3)})\cdots\\
&\qquad\qquad\qquad\qquad\qquad\qquad\cdots (M^3u_1^{-1}M^{-3})(M^2u_1M^{-2})(Mu_1^{-1}M^{-1})u_1,\\
\ro&=M^{g}M^{-1}u_1^{(-1)^{g-2}}M^{-1}u_1^{(-1)^{g-3}}M^{-1}\cdots M^{-1}u_1^{-1}M^{-1}u_1M^{-1}u_1^{-1}M^{-1}u_1.
  \end{aligned}
\] 
If $g$ is odd, by (B3') this gives
\[\begin{aligned}
   &\ro=\ro (M^{-1}u_1^{-1}M^{-1}u_1)^{\frac{g-1}{2}}\\
&(y^{-1}a_2a_3\cdots a_{g-1}ya_2a_3\cdots a_{g-1})^{\frac{g-1}{2}}=1.
  \end{aligned}
\]
If $g$ is even, we get
\[\begin{aligned}
   &\ro=M^{-1}u_1(M^{-1}u_1^{-1}M^{-1}u_1)^{\frac{g-2}{2}}\\
  &(y^{-1}a_2a_3\cdots a_{g-1}ya_2a_3\cdots a_{g-1})^{\frac{g-2}{2}}y^{-1}a_2a_3\cdots a_{g-1}=\ro.
  \end{aligned}
\]
\end{proof}
\begin{wn}
 Relation (B4) in the presentation given by Theorem \ref{ParSzep2} is superfluous.
\end{wn}
\begin{proof}
 As we saw in the proof of Proposition \ref{SimParSzep2}, relation (B4) is a consequence of relations 
\[\begin{aligned}
   &\ro a_i \ro^{-1}=a_i\quad\text{for $i=1,\ldots,g-1$,}\\
&\ro u_i \ro^{-1}=u_i^{-1}\quad\text{for $i=1,\ldots,g-1$},
  \end{aligned}
\]
and these relations are consequences of relations (A1), (A2), (B1), (B2), (C1)--(C5), (D) and (F).
\end{proof}
\begin{tw}\label{SimParSzep3}
If $g\geq 4$, then the group ${\cal M}(N_{g,0})$ is isomorphic to the quotient of the group ${\cal M}(N_{g,1})$ with
presentation given in Theorem \ref{SimParSzep1} obtained by adding a generator $\ro$ and relations
\begin{itemize}
 \item[(B3')] $(a_1a_2\cdots a_{g-1})^g=
 \begin{cases}
1&\text{for $g$ even}\\
\ro&\text{for $g$ odd,}
\end{cases}$
\item[(D1)] $\ro a_i=a_i\ro$\quad for $i=1,\ldots,g-1$,
\item[(D2)] $y\ro=\ro y^{-1}$,
\item[(E)] $\ro^2=1$,
\item[(F')] $(y\ro a_2a_3\cdots a_{g-1})^{g-1}=1$.
\end{itemize}
\end{tw}
\begin{proof}
We start by replacing relation (D') in the presentation given by Proposition \ref{SimParSzep2} by relations (D1) and (D2). 
Using these relations and relation (E) we rewrite (Fa) as
\[\begin{aligned}
   &(y\ro a_2a_3\cdots a_{g-1} y\ro a_2a_3\cdots a_{g-1})^{\frac{g-1}{2}}=1,\\
&(y\ro a_2a_3\cdots a_{g-1})^{g-1}=1.
  \end{aligned}
\]
Similarly, we rewrite (Fb) as
\[\begin{aligned}
   &(y \ro a_2a_3\cdots a_{g-1}y\ro a_2a_3\cdots a_{g-1})^{\frac{g-2}{2}}y\ro a_2a_3\cdots a_{g-1}=1, \\
&(y\ro a_2a_3\cdots a_{g-1})^{g-1}=1.
  \end{aligned}
\]
\end{proof}
\begin{uw}
 Observe that relations (B3') and (Fb),(F) allow to remove $\ro$ from the generating set, hence the presentations 
 of Proposition \ref{SimParSzep2} and Theorem~\ref{SimParSzep3} really uses Dehn twists and a crosscap slide as generators.
\end{uw}
\section{Geometric interpretation} \label{sec:geom}
The relations given by Theorem \ref{SimParSzep1} may look to be rather complicated,
so we devote this section to their geometric interpretation.

As we mentioned before, $y$ is a crosscap slide, that is it is the effect of pushing the crosscap $\mu_1$ along the curve $\alpha_1$ indicated in Figure \ref{r04}. 
\begin{figure}[h]
\begin{center}
\includegraphics[width=0.9\textwidth]{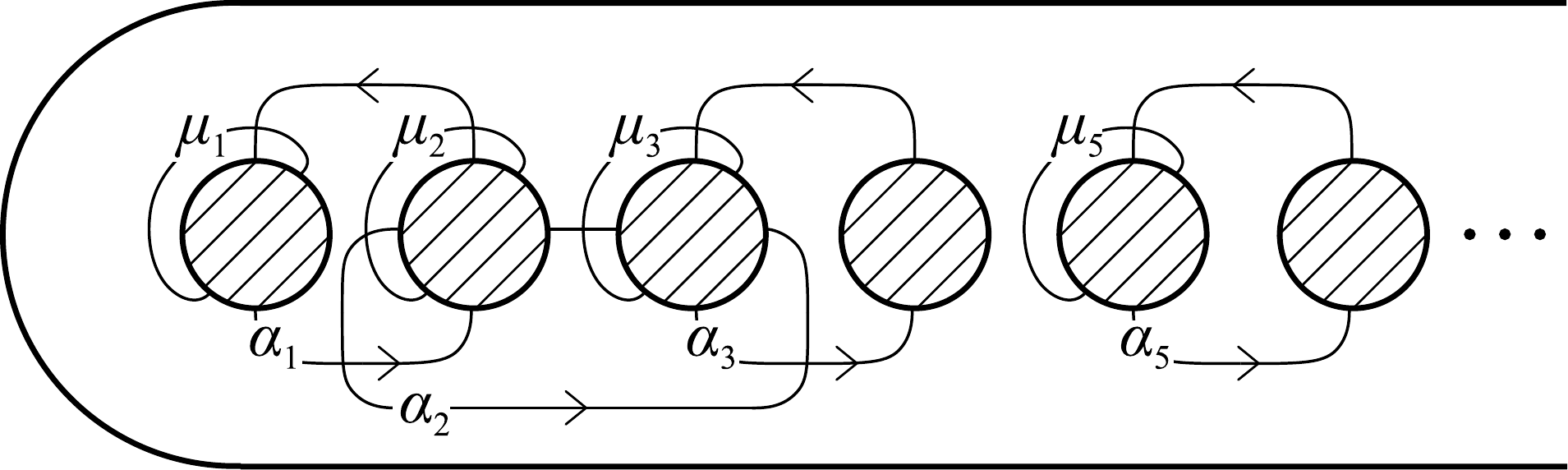}
\caption{Circles $\mu_i$ and $\alpha_i$.}\label{r04} %
\end{center}
\end{figure}
In general, crosscap slide $Y_{\mu,\alpha}$ is determined by an unoriented one-sided circle $\mu$ and an oriented
two-sided circle $\alpha$ which intersects $\mu$ in one point (see Figure \ref{r03} below and Section 2.2 of \cite{SzepParis}). 
\begin{figure}[h]
\begin{center}
\includegraphics[width=0.7\textwidth]{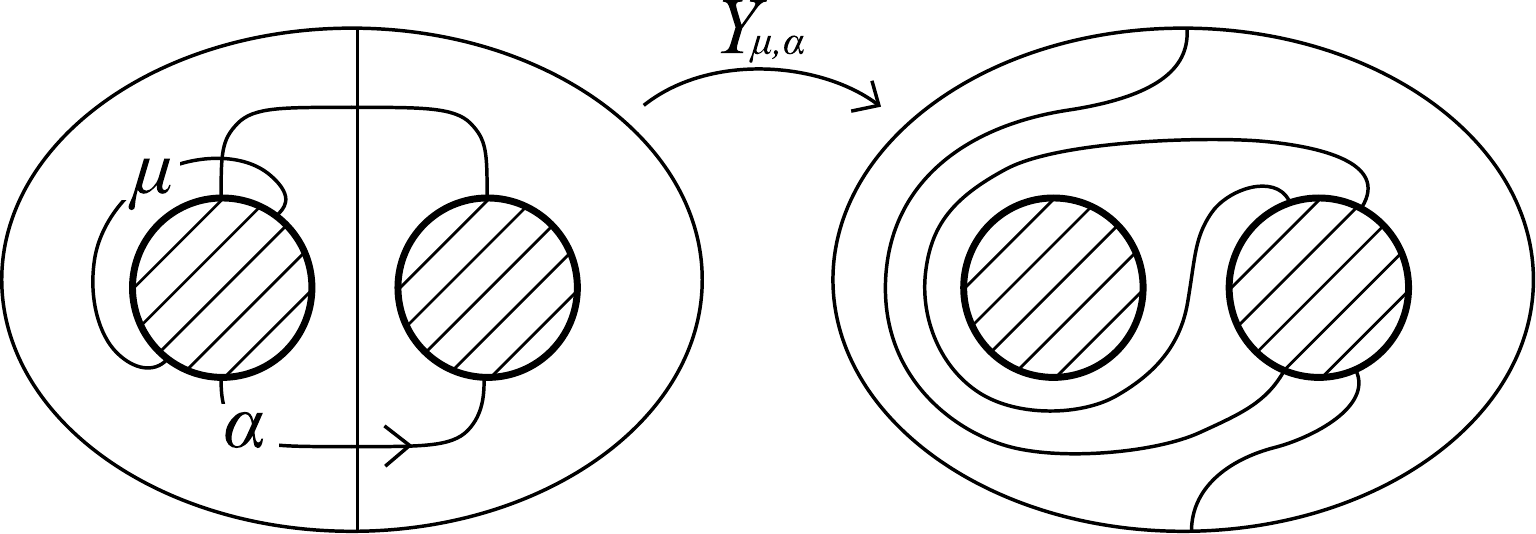}
\caption{Crosscap slide.}\label{r03} %
\end{center}
\end{figure}
Using this notation, for any 
$h\in {\cal{M}}(N)$ we have
\[hY_{\mu,\alpha}h^{-1}=Y_{h(\mu),h(\alpha)}.\]
Using this formula, it is straightforward to check that 
\[\begin{aligned}
   &a_2^{-1}ya_2=a_2^{-1}(Y_{\mu_1,\alpha_1})a_2=Y_{a_2^{-1}(\mu_1),a_2^{-1}(\alpha_1)}=Y_{\mu_1,\alpha_{1,3}},\\
   &a_3^{-1}a_2^{-1}ya_2a_3=a_3^{-1}(Y_{\mu_1,\alpha_{1,3}})a_3=Y_{\mu_1,\alpha_{1,4}},\\
   &a_4^{-1}a_3^{-1}a_2^{-1}ya_2a_3a_4=a_4^{-1}(Y_{\mu_1,\alpha_{1,4}})a_4=Y_{\mu_1,\alpha_{1,5}},\\
   &a_1a_2y^{-1}a_2^{-1}a_1^{-1}=a_1a_2(Y_{\mu_1,-\alpha_1})a_2^{-1}a_1^{-1}=Y_{\mu_2,\alpha_2},\\
   &a_2a_3a_1a_2ya_2^{-1}a_1^{-1}a_3^{-1}a_2^{-1}=a_2a_3(Y_{\mu_2,-\alpha_2})a_3^{-1}a_2^{-1}=Y_{\mu_3,\alpha_3},\\
   &a_4a_5a_3a_4a_2a_3a_1a_2ya_2^{-1}a_1^{-1}a_3^{-1}a_2^{-1}a_4^{-1}a_3^{-1}a_5^{-1}a_4^{-1}=Y_{\mu_5,\alpha_5}.   
  \end{aligned}
\]
for circles $\mu_i,\alpha_i,\alpha_{1,i}$ indicated in Figures \ref{r04} and \ref{r07}. 
\begin{figure}[h]
\begin{center}
\includegraphics[width=1\textwidth]{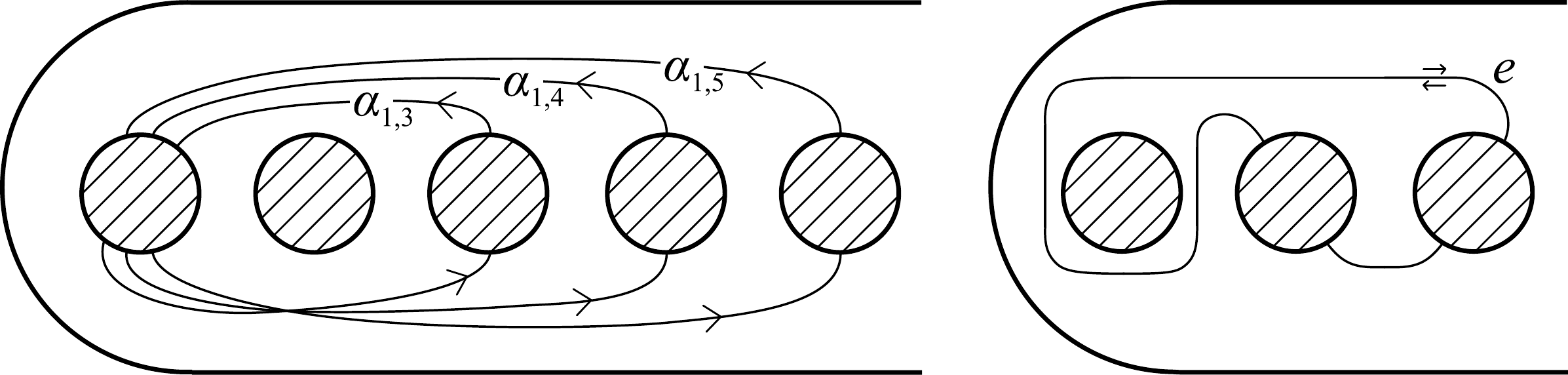}
\caption{Circles $\alpha_{1,i}$ and $e=ya_2^{-1}y^{-1}$.}\label{r07} %
\end{center}
\end{figure}

In particular relations (B1') and (C7') are 
the commutativity relations between respectively $y,Y_{\mu_3,\alpha_3}$ and $b,Y_{\mu_5,\alpha_5}$.
Relation (C1') is the obvious commutativity between $y$ and $a_i$ for $i\geq 3$, and relation (C2') is
the commutativity between $a_2$ and the twist $e=ya_2^{-1}y^{-1}$ indicated in Figure \ref{r07}. 

The meaning of (C4') is also obvious: the crosscap slide $y$ maps $\alpha_1$ to the same circle, but changes 
the local orientation of its regular neighbourhood.

Before we explain the geometric meaning of (B2'), we rewrite this relation using (A1)--(A2), (C2') and (C4'):
\[\begin{aligned}
&\SL{y}(a_2a_1y^{-1}a_2^{-1}ya_1a_2)y=a_1(a_2\SR{a_1}y^{-1}a_2^{-1}y\SL{a_1}a_2)\SR{a_1},\\
&(a_2a_1y^{-1}a_2^{-1}ya_1a_2)y\SL{a_1^{-1}}=y^{-1}a_1(a_2y^{-1}\SR{a_1^{-1}}a_2^{-1}a_1^{-1}ya_2),\\
&(a_1a_2\SR{a_1}y^{-1}a_2^{-1}y\SL{a_1}a_2)y=y^{-1}(a_1a_2y^{-1}a_2^{-1}a_1^{-1})(a_2^{-1}ya_2),\\
   &(a_1a_2y^{-1}a_2^{-1}a_1^{-1})(a_2^{-1}ya_2)y=y^{-1}(a_1a_2y^{-1}a_2^{-1}a_1^{-1})(a_2^{-1}ya_2),\\
   &Y_{\mu_2,\alpha_2}Y_{\mu_1,\alpha_{1,3}}y=y^{-1}Y_{\mu_2,\alpha_2}Y_{\mu_1,\alpha_{1,3}}.
  \end{aligned}
\]
Geometrically, $Y_{\mu_2,\alpha_2}Y_{\mu_1,\alpha_{1,3}}$ has the effect of pushing two first crosscaps 
through the third one. It is clear that this action maps $\alpha_1$ to the same circle but with 
the reversed orientation.

As for (C8') it can be rewritten as
\[\begin{aligned}
   &Y_{\mu_1,\alpha_{1,5}}Y^{-1}_{\mu_1,\alpha_{1,4}}Y_{\mu_1,\alpha_{1,3}}Y^{-1}_{\mu_1,\alpha_{1}}=
   b_{1}b_2^{-1},
  \end{aligned}
\]
where $b_1$ and $b_2$ are twists about the circles $\beta_1,\beta_2$ indicated in Figure \ref{r05}. 
\begin{figure}[h]
\begin{center}
\includegraphics[width=0.73\textwidth]{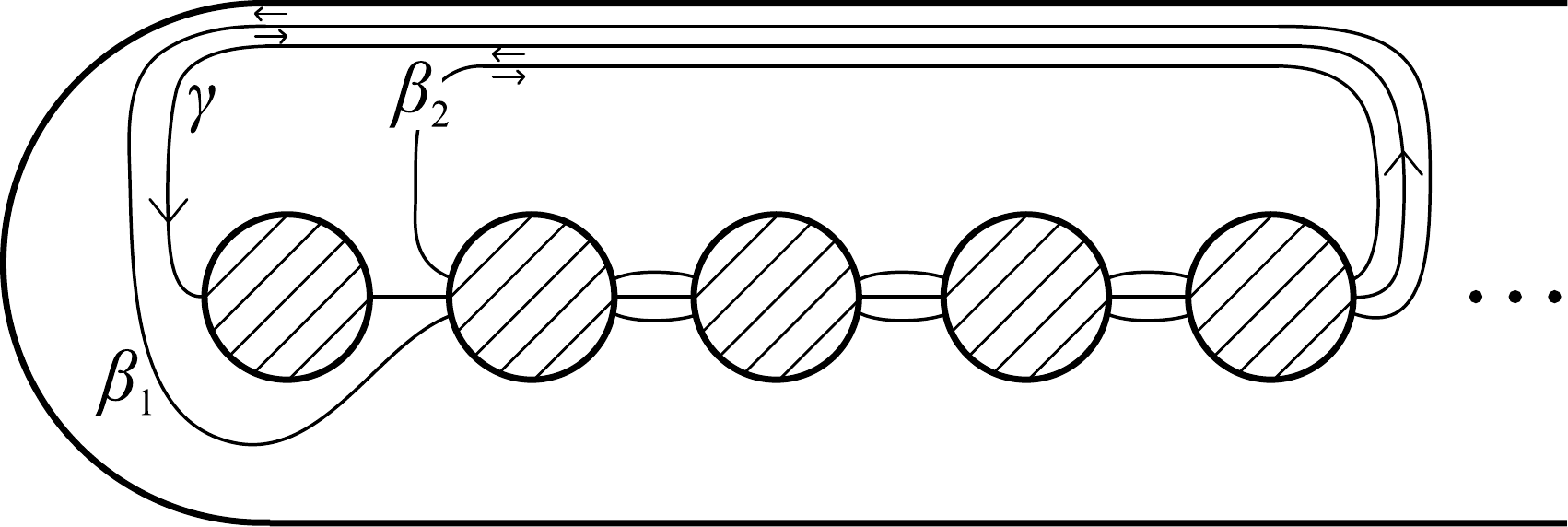}
\caption{Circles $\beta_1,\beta_2$ and $\gamma$.}\label{r05} %
\end{center}
\end{figure}
The composition 
on the left-hand side is equal to the crosscap slide $Y_{\mu_1,\gamma}$, hence (C8') is simply the well-known 
relation between a point/crosscap push along $\gamma$ and the boundary twists of a regular neighbourhood of 
$\gamma$ (see Lemma 2.2 of \cite{SzepParis}).

Finally, we turn to relation (C6'), which is really a relation between four Dehn twists
\[bb'=aa',\]
where $b,b',a,a'$ are twists about the circles $\beta,\beta',\alpha,\alpha'$ indicated in Figure  \ref{r06}. 
\begin{figure}[h]
\begin{center}
\includegraphics[width=1\textwidth]{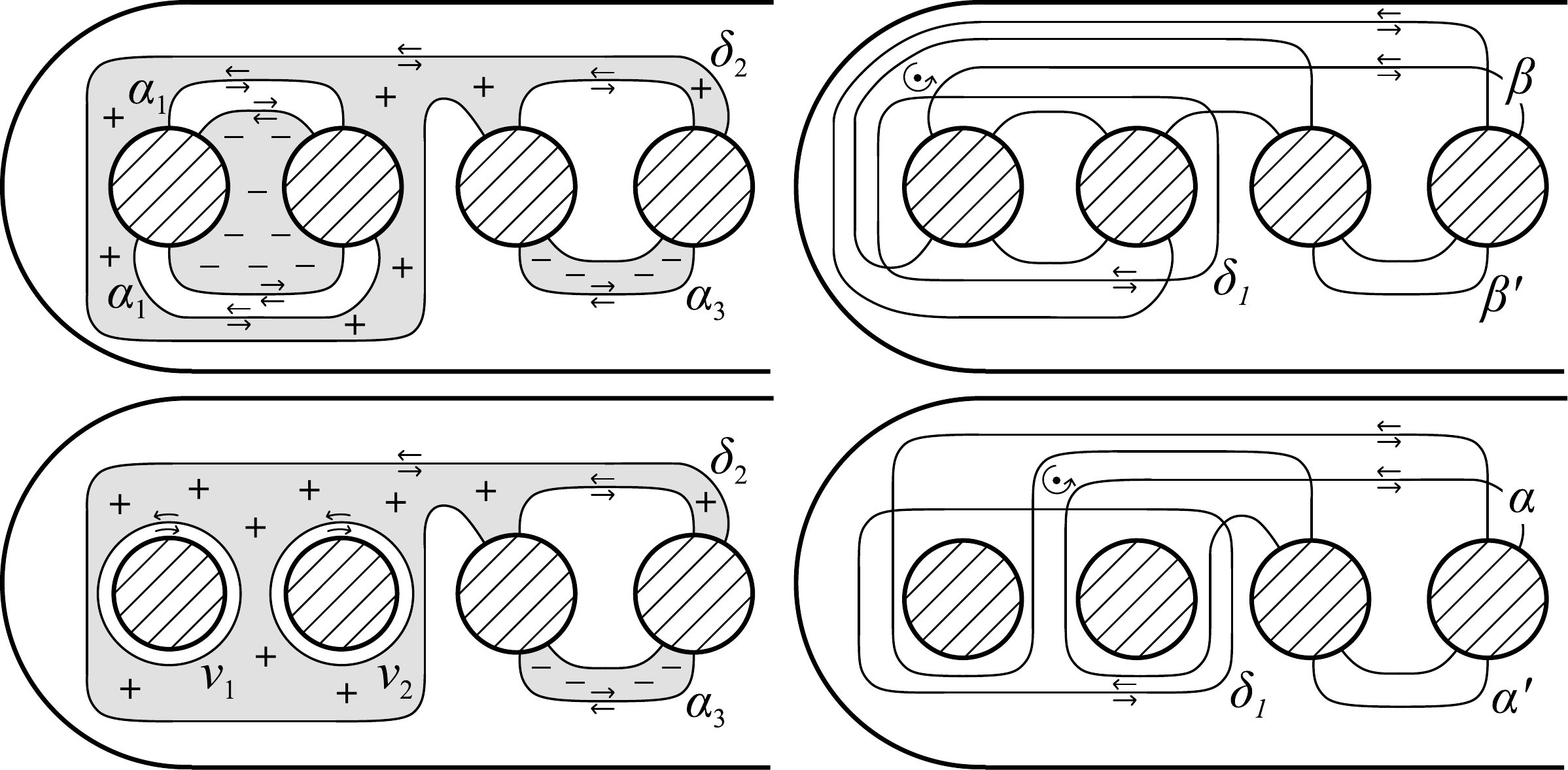}
\caption{Lantern relations $d_2a_3a_1a_1^{-1}=d_1bb'$ and $d_2a_3v_2v_1=d_1aa'$.}\label{r06} %
\end{center}
\end{figure}

It turns out that this relation
is a consequence of two lantern relations indicated in Figure~\ref{r06}. On the left side of this 
figure we indicated two sets of four circles, such that circles in each of these sets bound 
a sphere with four holes (shaded regions in this figure). Moreover, the plus and minus signs in these figures indicate 
the positive orientations of these spheres, that is the orientations with respect to which we consider twists to be 'the right-handed Dehn twists'. Hence, there are two lantern relations 
\[\begin{aligned}
   &d_2a_3a_1a_1^{-1}=d_1bb',\\
   &d_2a_3v_2v_1=d_1aa',
  \end{aligned}
\]
where $d_1,d_2,v_1,v_2$ are twists about circles $\delta_1,\delta_2,\nu_1,\nu_2$ respectively. Note that the order of twists on the right-hand side of the above relations is 
determined by the anticlockwise direction, which we indicated by the small arrows on the right side of Figure \ref{r06}. Since 
$v_1$ and $v_2$ are trivial, these two lanterns imply (C6').

\bibliographystyle{abbrv}
\bibliography{SimpSzepParis.bbl}
\end{document}